\documentclass[12pt]{amsart}

\usepackage[margin=1in]{geometry}
\usepackage{amssymb,mathtools,booktabs}
\usepackage[hidelinks]{hyperref}

\makeatletter
\@namedef{subjclassname@2020}{%
  \textup{2020} Mathematics Subject Classification}
\makeatother

\newtheorem{theorem}{Theorem}[section]
\newtheorem{proposition}[theorem]{Proposition}
\newtheorem{lemma}[theorem]{Lemma}
\newtheorem{corollary}[theorem]{Corollary}
\theoremstyle{definition}
\newtheorem{definition}[theorem]{Definition}
\theoremstyle{remark}

\newcommand{\PP}{\mathbf P}
\newcommand{\Aaff}{\mathbf A}
\newcommand{\GW}{\operatorname{GW}}
\newcommand{\Tr}{\operatorname{Tr}}
\newcommand{\Bl}{\operatorname{Bl}}
\newcommand{\Sym}{\operatorname{Sym}}
\newcommand{\Tact}{\operatorname{Tact}}

\newcommand{\OO}{\mathcal O}
\newcommand{\HHyp}{\mathbb H}
\newcommand{\Aone}{\mathbb A^1}
\newcommand{\kbar}{\overline{k}}
\newcommand{\RR}{\mathbb R}
\newcommand{\CC}{\mathbb C}
\newcommand{\QQ}{\mathbb Q}
\newcommand{\Disc}{\operatorname{Disc}}
\newcommand{\inertia}{\operatorname{inertia}}
\newcommand{\TaQ}{\operatorname{TaQ}}
\newcommand{\Hom}{\operatorname{Hom}}
\numberwithin{equation}{section}

\title[Two-flag degenerations of tangent circles]
{Two-flag degenerations and real circles tangent to three conics}

\author{Haoyang Liu}
\address{Department of Mathematics, University of California, Santa Barbara,
California 93106, USA}
\email{haoyangliu@ucsb.edu}
\author{Tianle Liu}
\address{University of Southern California}
\email{tianleli@usc.edu}
\author{Guorui Xu}
\address{Fudan University, Shanghai, China}
\email{xugr@fudan.edu.cn}
\author{Fan Yang}
\address{University of Southern California}
\email{fyang399@usc.edu, fanyang399@gmail.com}


\begin{document}

\begin{abstract}
Three general plane conics admit \(184\) complex tangent circles, and it had
been conjectured that at most \(136\) of them could be real.  We construct an
explicit strongly general triple of smooth conics over \(\QQ\) with exactly
\(160\) real tangent circles; the same count therefore occurs on a nonempty
Euclidean chamber.  The construction combines a fourfold splitting theorem for
two flagged double-line degenerations with exact Sturm--Tarski, elimination,
and interval certificates.  We also show that the Grothendieck--Witt-valued
count is \(92\mathbb H\) and express its real local signs, up to a fixed
orientation convention, in terms of curvature differences, residual
intersection divisors, and contact normals.  The exact-arithmetic code and
certificate data are archived in the accompanying repository.
\end{abstract}

\maketitle
\tableofcontents

\section{Introduction}

Breiding, Lindberg, Ong, and Sommer proved that three general complex plane
conics have \(184\) tangent circles.  They constructed a real triple with
\(136\) real tangent circles and conjectured that this number was maximal
\cite{BLOS}.  Our first result disproves this conjecture and shows that the real
count is not controlled by the corresponding oriented count.

The construction in \cite[Theorem~2.2]{BLOS} degenerates all three prescribed
conics to flagged double lines.  The resulting mixed point--line problem has
seventeen real solutions, and smoothing the three flags produces
\(2^3\cdot17=136\) real circles.  We instead degenerate two conics and keep the
third smooth.  The reduced problem then has \(46\) complex solutions, of which
\(40\) are real, and every solution splits into four.

The limiting data are naturally described on the space of complete conics.
The exceptional fiber over a double line parametrizes effective divisors of
degree two on its support; a flagged double line is the diagonal case, in
which the divisor is a repeated point.  For such a degeneration, tangency
limits to one of two conditions: the circle passes through the flag point, or
it is tangent to the supporting line.  With two flags the reduced circle
problem separates into five rational branches.

Write \(M_{\RR}\) for the maximum number of real tangent circles among the
strongly general real triples of Section~\ref{sec:complete-conics}.  The real
count is constant on Euclidean chambers, but not necessarily on a real
Zariski-open set.

\subsection{Main results}

\begin{theorem}\label{thm:introduction-real}
There is an explicit strongly general triple of smooth plane conics over
\(\QQ\) with exactly \(160\) real tangent circles.  All these circles have
positive squared radius.  The triple lies in a nonempty Euclidean open set on
which the same assertions hold.  In particular,
\[
        160\leq M_{\RR}\leq184.
\]
\end{theorem}

The numbers of complex roots on the five limiting branches are
\[
        6,\qquad12,\qquad12,\qquad8,\qquad8,
\]
while the numbers of real points are
\[
        6,\qquad10,\qquad12,\qquad4,\qquad8.
\]
Each reduced point has local length four in the unreduced special fiber and
splits into four simple points under a transverse smoothing.  At each of the
\(40\) real points, the two normal coefficients have the sign for which all
four branches are real.  The six nonreal points therefore give \(24\) nonreal
geometric branches.  This argument does not use numerical continuation.

The fourfold splitting is valid for every transverse two-flag degeneration.
In suitable local coordinates, the two degenerating tact equations have the
form
\[
 u^2U_1(u,v)+\delta^2A_1(u,v)+\delta^4B_1(u,v,\delta^2)=0,
 \qquad
 v^2U_2(u,v)+\delta^2A_2(u,v)+\delta^4B_2(u,v,\delta^2)=0.
\]
The special local algebra is
\[
        \CC[[u,v]]/(u^2,v^2),
\]
and the substitutions \(u=\delta x\), \(v=\delta y\) exhibit the four
branches.

For comparison, Bachmann and Wickelgren treated Chasles's problem for five
plane conics on the space of complete conics and obtained the class
\(1632\HHyp\) \cite{BachmannWickelgrenExcess}.  Our quadratic count uses the
same compactified geometry, but two tact summands are replaced by the two
hyperplane conditions cutting out the linear system of circles.

We now turn to the quadratic refinement.  Let \(P\) be the projective space
of conics, let \(S\subset P\) be the Veronese surface of
double lines, and put \(X=\Bl_SP\).  If \(H\) is the pullback of the
hyperplane class and \(E\) is the exceptional divisor, then the proper
transform of a tact divisor has class \(T=6H-2E\).  The two circle conditions
and three tangency conditions define a section of
\[
 \mathcal E=\OO_X(H)^{\oplus2}\oplus\OO_X(T)^{\oplus3}.
\]

\begin{theorem}\label{thm:introduction-quadratic}
Let \(k\) be a field of characteristic zero, and let three smooth plane conics
over \(k\) be strongly general.  With the relative orientation fixed in
Proposition \ref{prop:orientation}, the sum of the local indices of their
tangent circles is
\[
        92\HHyp\quad\text{in }\GW(k).
\]
Its rank is \(184\).
\end{theorem}

Over \(\RR\), the signature of this class is zero.  The local sign also admits
an intrinsic description.  If \(C\) is tangent to \(Q_i\) at \(p_i\), let
\(n_i\) be the outward unit normal of the circle and measure both curvatures
with respect to this normal.  Up to the fixed global convention, the local
sign is the sign of
\[
 \prod_{i=1}^3\delta_i(\kappa_{Q_i}-\kappa_C)
 \det\begin{pmatrix}
 n_{1x}&n_{1y}&1\\
 n_{2x}&n_{2y}&1\\
 n_{3x}&n_{3y}&1
 \end{pmatrix}.
\]
Here \(\delta_i\) is the discriminant square class of the residual
degree-two divisor in \(C\cap Q_i\); over \(\RR\), its sign is positive when
the two residual points are real and negative when they are conjugate
nonreal points.  Thus the enriched count records the residual intersection
type and the differential geometry of the three contacts.

The finite real tangent-circle algebra carries two natural trace forms.
The ordinary trace form records the number of real circles.  The form weighted
by the oriented Jacobian records the Poincar\'e--Hopf class.  In a chamber with
\(R\) real circles and \(c=(184-R)/2\) nonreal conjugate pairs, their inertias
are
\[
        (R+c,c)\qquad\text{and}\qquad(92,92).
\]
Consequently the former inequality \(R\leq136\) is equivalent to the ordinary
trace form having at least \(24\) negative directions.  It is a reformulation
of the proposed bound, not a second obstruction.

\subsection{Outline of the paper}

In Section \ref{sec:complete-conics}, we establish the quadratic count.  In
Section \ref{sec:degeneration}, we prove the fourfold splitting theorem.
Sections \ref{sec:construction} and \ref{sec:certificate} construct the
\(160\)-circle chamber and certify a rational point in it.  Section
\ref{sec:local-indices} identifies the local indices and trace forms.  The
accompanying repository contains the exact-arithmetic checkers and frozen
certificate data.  We do not prove that \(160\) is maximal.

\section*{Acknowledgements}

The first author would like to thank Yeqin Liu, Gabriel Ong, and Kirsten
Wickelgren for helpful discussions.  The \(160\)-circle counterexample was
found with assistance from the generative-AI research agent DANUS.
Generative-AI tools were also used for exploratory code, proof checking,
pre-submission review, and language editing.  No AI output is used as evidence
for a mathematical assertion.  All proof-critical exact computations are
described in the paper, and their code and frozen data are archived in the
accompanying repository.  The authors have verified the text and computations
and accept responsibility for the entire contents.

\section{Complete conics and the quadratic count}\label{sec:complete-conics}

Throughout this section, we work over a field \(k\) of characteristic zero.
Let \(V\) be a
three-dimensional vector space and put
\[
 P=\PP(\Sym^2V^\vee)\cong\PP^5_k.
\]
The Veronese surface \(S\subset P\) parametrizes double lines.  We write
\[
        X=\Bl_SP,\qquad \pi:X\longrightarrow P,
\]
and denote the exceptional divisor by \(E\).  If \(H=\pi^*c_1(\OO_P(1))\),
the blowup formula gives
\begin{equation}\label{eq:canonical}
        \omega_X\cong\OO_X(-6H+2E).
\end{equation}

For a smooth conic \(Q\), let \(\Tact_Q\subset P\) be the hypersurface of
conics tangent to \(Q\), and let \(\widetilde{\Tact}_Q\) be its proper
transform on \(X\).  We use the following normalization throughout.  Choose a
symmetric matrix \(M_Q\) representing \(Q\).  If \(N\) represents the variable
conic, then
\[
 \Disc_\lambda\det(M_Q+\lambda N)
\]
is the degree-six equation of \(\Tact_Q\).  The standard normal-coordinate
calculation for the cubic pencil discriminant shows that this equation belongs
to \(I_S^2\), where \(I_S\) is the ideal of the Veronese surface; see
\cite[Section~3]{BachmannWickelgrenExcess}.  We record the part of the
calculation that proves that the generic multiplicity is not larger.  Over an
algebraic closure all smooth plane conics form a single
\(\operatorname{PGL}_3\)-orbit, so the order along \(S\) may be checked after
putting a fixed smooth \(Q\) in convenient coordinates.  Let
\(L^2\in S\) be a generic double line.  Thus \(L\) is not tangent to \(Q\).
After passing to an algebraic closure, choose coordinates with
\(L=\{X_2=0\}\) and with the restriction of \(M_Q\) to \(L\) equal to
\(I_2\).  For a normal direction
\[
 A=\begin{pmatrix}u&v\\ v&w\end{pmatrix},
 \qquad
 N_s=\begin{pmatrix}sA&0\\0&1\end{pmatrix},
\]
direct expansion gives
\[
 \Disc_\lambda\det(M_Q+\lambda N_s)
 \equiv s^2\bigl((u-w)^2+4v^2\bigr)\pmod {s^3}.
\]
The displayed quadratic form is nonzero.  Consequently the tact equation
lies in \(I_S^2\setminus I_S^3\) at the generic point of \(S\), and its order
along \(S\) is exactly two.  Hence the total transform contains \(E\) with
multiplicity exactly two.  Division
by the square of a local equation of \(E\) therefore defines a section
\(\tau_Q\) whose zero divisor is \(\widetilde{\Tact}_Q\).  Thus
\begin{equation}\label{eq:tact-class}
        \OO_X(\widetilde{\Tact}_Q)\cong\OO_X(6H-2E).
\end{equation}
Put \(T=6H-2E\).  Replacing
\(M_Q\) by \(uM_Q\) multiplies \(\tau_Q\) by \(u^6\).  The divisor and every
local square class below are therefore independent of the projective
representative of \(Q\).

Choose homogeneous coordinates \([X_0:X_1:X_2]\) on \(\PP(V)\).  The
algebraic circles are the conics
\[
 a(X_0^2+X_1^2)+dX_0X_2+eX_1X_2+fX_2^2=0.
\]
They are cut out in \(P\) by two independent linear equations.  Let
\(\lambda_1,\lambda_2\in H^0(X,\OO_X(H))\) be their pullbacks.  For a triple
\(Q_\bullet=(Q_1,Q_2,Q_3)\), use the discriminant-normalized sections just
defined and set
\begin{equation}\label{eq:global-section}
 \sigma_{Q_\bullet}=(\lambda_1,\lambda_2,
 \tau_{Q_1},\tau_{Q_2},\tau_{Q_3})
 \in H^0(X,\mathcal E),
 \qquad
 \mathcal E=\OO_X(H)^{\oplus2}\oplus\OO_X(T)^{\oplus3}.
\end{equation}

For the family version, let
\(B_0\subset\PP(\Sym^2V^\vee)\) be the open subscheme of smooth conics.  The
discriminant is homogeneous of degree six in the coefficients of the
prescribed conic.  Writing \(q_X\) and \(q_0\) for the projections from
\(X\times B_0\), it therefore defines a universal tact section
\[
 \boldsymbol\tau\in
 H^0\!\left(X\times B_0,
 q_X^*\OO_X(T)\otimes q_0^*\OO_{B_0}(6)\right).
\]
For \(B=B_0^3\), let \(p_X:X\times B\to X\), \(p_B:X\times B\to B\), and
\(\operatorname{pr}_i:B\to B_0\) denote the projections.  The two fixed
circle equations and the three pullbacks of \(\boldsymbol\tau\) form a section
of
\begin{equation}\label{eq:universal-bundle}
 p_X^*\OO_X(H)^{\oplus2}\oplus
 \bigoplus_{i=1}^3
 \left(p_X^*\OO_X(T)\otimes
 p_B^*\operatorname{pr}_i^*\OO_{B_0}(6)\right).
\end{equation}
Let \(\mathcal Z\subset X\times B\) denote its zero scheme.  Its fiber over
\(Q_\bullet\) is the zero scheme of \(\sigma_{Q_\bullet}\): choosing matrix
representatives for the \(Q_i\) only trivializes the parameter factors and
rescales the corresponding equations by nonzero scalars.

Let \(Y=V(\lambda_1,\lambda_2)\subset X\), and let
\(Y^\circ\subset Y\setminus E\) be the open set of smooth circles.  The next
lemma isolates the boundary point that is implicit in the complete-conic
enumeration.

\begin{lemma}\label{lem:no-boundary-domination}
Let \(B_0\subset\PP(\Sym^2V^\vee)\) be the open set of smooth conics, and let
\[
 \mathcal T=\{(x,Q)\in Y\times B_0:\tau_Q(x)=0\}.
\]
For every geometric point \(x\in Y\), the fiber \(\mathcal T_x\) is a proper
closed subset of \(B_0\), of dimension at most four.  Consequently, if
\(W=Y\setminus Y^\circ\), then the locus
of \((x,Q_1,Q_2,Q_3)\in W\times B_0^3\) satisfying
\(\tau_{Q_i}(x)=0\) for all \(i\) has dimension at most fourteen.  In
particular, its image does not dominate \(B_0^3\).
\end{lemma}

\begin{proof}
We work over an algebraic closure.  If \(x\notin E\), it represents a conic
of rank two or three.  After a linear change of coordinates, its matrix is
\(N=\operatorname{diag}(1,1,0)\) or
\(N=\operatorname{diag}(1,1,1)\), respectively.  In the first case, take the
smooth conic with matrix \(M=\operatorname{diag}(1,2,1)\); then
\[
 \det(M+\lambda N)=(1+\lambda)(2+\lambda)
\]
viewed as a cubic with zero \(\lambda^3\)-coefficient, has universal cubic
discriminant \(1\).  In the second case, the choice
\(M=\operatorname{diag}(1,2,3)\) gives
\[
 \det(M+\lambda N)=(1+\lambda)(2+\lambda)(3+\lambda),
\]
whose discriminant is \(4\).  Hence \(Q\mapsto\tau_Q(x)\) is not the zero
polynomial.

It remains to check the exceptional fiber.  The only double line in the
circle system is the line at infinity counted twice.  Write its equation as
\(X_2^2\).  On the affine chart in which the coefficient of \(X_2^2\) is one,
the tangent space to the Veronese surface at \(X_2^2\) is represented by the
terms divisible by \(X_2\).  The normal space is therefore the space of binary
quadratic forms in \(X_0,X_1\).  Thus a point of the exceptional fiber is
represented by a nonzero symmetric matrix
\[
 A=\begin{pmatrix}u&v\\v&w\end{pmatrix}
\]
in the normal directions to the Veronese surface.  If \(M_\infty\) is the
upper-left \(2\)-by-\(2\) block of a matrix \(M\) for \(Q\), put
\[
 b=\det(M_\infty),
 \qquad q=um_{22}+wm_{11}-2vm_{12}.
\]
The normal slice through this point is represented by
\[
 N_s=\begin{pmatrix}su&sv&0\\sv&sw&0\\0&0&1\end{pmatrix}
\]
and \(s\) is a local equation for the exceptional divisor on the corresponding
blowup chart.  Substitution in the determinant-pencil discriminant gives
\begin{equation}\label{eq:exceptional-tact-restriction}
 \Disc_\lambda\det(M+\lambda N_s)
 \equiv s^2b^2\bigl(q^2-4\det(A)b\bigr)\pmod {s^3}.
\end{equation}
The coefficient of \(s^2\) is the degree-two initial form along the Veronese
surface, hence the value of the proper tact section in the direction \([A]\).
For every nonzero \(A\), it is a nonzero polynomial in
\(M_\infty\).  Indeed, diagonalize \(A\) and take
\(M_\infty=\operatorname{diag}(x,y)\).  If
\(A=\operatorname{diag}(u,0)\), the displayed coefficient is nonzero whenever
\(uxy\ne0\).  If \(A=\operatorname{diag}(u,w)\) is invertible, it equals
\[
       (xy)^2(uy-wx)^2,
\]
so it is nonzero for suitable \(x,y\ne0\).  Such a block extends, for example
by a nonzero diagonal last entry, to a smooth matrix \(M\).  Thus the tact
fiber is proper at every point of \(Y\cap E\).

The singular circles form a surface in the three-dimensional circle system,
and \(Y\cap E\) is contained in the two-dimensional exceptional fiber over its
unique double line.  Thus \(\dim W\leq2\).  Since every tact fiber in \(B_0\)
has dimension at most four, the triple incidence over \(W\) has dimension at most
\(2+3\cdot4=14\), whereas \(\dim B_0^3=15\).
\end{proof}

\begin{lemma}\label{lem:ordinary-discriminant}
Let \(C\) be a smooth conic over a field of characteristic zero.  In the
projective space of nonzero sections of \(\OO_C(2)\), let
\(\Delta_C^{\mathrm{ord}}\) be the locus of sections whose zero divisor has
the form
\[
                 2p+D,
\]
where \(D\) is reduced of degree two and is disjoint from \(p\).  The
incidence of pairs \(([F],p)\) with this property maps isomorphically onto
\(\Delta_C^{\mathrm{ord}}\).  Moreover, \(\Delta_C^{\mathrm{ord}}\) is a
 smooth reduced open subspace of the discriminant divisor.  Consequently,
 for variable plane quadrics \(Q\), the ordinary contact incidence with \(C\)
 maps scheme-theoretically isomorphically onto the ordinary open subspace of
 the tact divisor.  The complement of this open subspace is the singular
 locus of the discriminant divisor.
\end{lemma}

\begin{proof}
The assertions may be checked after an extension of the ground field.  Use
\(C\cong\PP^1\), choose an affine coordinate \(x\) for which none of the four
zeros of a fixed section \(F_0\) lies at infinity, and divide by the leading
coefficient.  Thus \(F_0\) is a monic quartic with a unique double zero \(p\)
and two simple residual zeros.  Write
\[
                 F_0=f_0g_0,
\]
where \(f_0=(x-p)^2\) and \(g_0\) is the monic residual quadratic.  Since
\(\operatorname{Res}(f_0,g_0)\ne0\), the differential of the multiplication
map
\[
 \Aaff^2\times\Aaff^2\longrightarrow\Aaff^4,
 \qquad (f,g)\longmapsto fg,
\]
on monic quadratic polynomials is invertible at \((f_0,g_0)\); its
determinant is \(\operatorname{Res}(f_0,g_0)\).  After an
\(\acute{e}\)tale localization, every nearby monic quartic has a
unique factorization \(F=fg\) deforming \(F_0=f_0g_0\).

Write \(f=x^2+ux+v\).  The root divisor of \(f\) is nonreduced precisely
when \(u^2-4v=0\), and on that divisor its repeated root is \(-u/2\).
The product formula
\[
 \Disc(fg)=\Disc(f)\Disc(g)\operatorname{Res}(f,g)^2
\]
shows that the quartic discriminant is a unit times \(u^2-4v\) near
\(F_0\).  It is therefore reduced and smooth there, and the repeated root
is a regular function on it.  This proves that the incidence projection is
an isomorphism over the stated open set.  The restriction map
\[
 H^0(\PP^2,\OO_{\PP^2}(2))\longrightarrow H^0(C,\OO_C(2))
\]
 is surjective with kernel \(\langle C\rangle\).  Away from that kernel its
 projectivization is smooth, so pulling back the preceding description gives
 the assertion about variable plane quadrics.  Over an algebraic
 closure the singular locus of the discriminant of a binary quartic consists
 precisely of forms having either a root of multiplicity at least three or
 two distinct multiple roots.  These are exactly the nonordinary divisors in
 the statement, which proves the last assertion.
\end{proof}

We shall say that the triple is strongly general if, after extension to \(\kbar\), the zero
scheme of \(\sigma_{Q_\bullet}\) is finite and \(\acute{e}\)tale, is disjoint
from \(E\), and consists of smooth circles.  We further require that, for
every zero circle \(C\) and every \(i\), the intersection \(C\cap Q_i\) have
exactly one contact of multiplicity two and only transverse residual
intersections.  We refer to this condition as ordinary tangency.

\begin{lemma}\label{lem:strongly-general}
There is a nonempty Zariski-open subset
\(B^{\mathrm{sg}}\) of the parameter space of triples of smooth conics such
that every triple in \(B^{\mathrm{sg}}\) is strongly general.  For every
strongly general triple the
tangent-circle scheme has geometric length
\[
        H^2(6H-2E)^3=184.
\]
When \(k=\RR\), the open set \(B^{\mathrm{sg}}\) is defined over \(\RR\), and
its real points are
Euclidean dense in the real parameter space.
\end{lemma}

\begin{proof}
Put \(B=B_0^3\).  Over \(Y^\circ\times B_0\), consider the incidence of
triples \((C,Q,p)\) for which \(p\in C\cap Q\) and
\(T_pC=T_pQ\).  The open subspace on which the contact has multiplicity two,
\(p\) is the unique nontransverse point of \(C\cap Q\), and the residual
intersections are transverse maps isomorphically to an open subspace
\(\mathcal T^{\mathrm{ord}}\) of the universal tact divisor by Lemma
\ref{lem:ordinary-discriminant}.

For a smooth conic \(C\), restriction of quadratic forms gives an exact
sequence
\[
 0\longrightarrow \langle C\rangle\longrightarrow
 H^0\!\left(\PP^2,\OO_{\PP^2}(2)\right)
 \longrightarrow H^0\!\left(C,\OO_C(2)\right)\longrightarrow0.
\]
After identifying \(C\) with \(\PP^1\), the last term is
\(H^0(\PP^1,\OO_{\PP^1}(4))\).  This complete linear system separates
zero-dimensional subschemes of length at most four.  Thus requiring
\(Q|_C\) to vanish to order at least two at a prescribed point imposes two
independent linear conditions on \(Q\), while order at least three imposes
three.  At two prescribed distinct points, requiring order at least two at
both imposes four independent conditions.

Allowing one contact point to vary shows that tangency is a divisor in the
\(Q\)-space and that contact of order at least three has codimension at least
two.  Allowing two distinct contact points to vary gives the same bound for
the locus of two contacts.  The kernel \(\langle C\rangle\) corresponds to
\(Q=C\); as \(C\) varies, this diagonal has dimension three in the
eight-dimensional space \(Y^\circ\times B_0\), and therefore does not affect
the estimate.  Deleting singular \(Q\) and imposing transverse residual
intersections are open conditions.  Hence the nonordinary locus has
codimension at least two in \(Y^\circ\times B_0\), and the same jet-surjectivity
shows that \(\mathcal T^{\mathrm{ord}}\to Y^\circ\) is smooth of relative
dimension four.

Taking the fiber product of three copies of
\(\mathcal T^{\mathrm{ord}}\) over \(Y^\circ\) gives a smooth incidence
\[
 \mathcal I^{\mathrm{ord}}
 =\mathcal T^{\mathrm{ord}}
  \mathop{\times}_{Y^\circ}\mathcal T^{\mathrm{ord}}
  \mathop{\times}_{Y^\circ}\mathcal T^{\mathrm{ord}}
 \longrightarrow B
\]
whose source and target both have dimension fifteen.  Compare it with
the universal zero scheme \(\mathcal Z\subset X\times B\) defined by
\eqref{eq:universal-bundle}.  The projection \(\mathcal Z\to B\) is proper.
Lemma \ref{lem:no-boundary-domination} shows that the part over
\(Y\setminus Y^\circ\) has image in a proper closed subset of \(B\).  The jet
calculation above gives the same conclusion for the closure of the
nonordinary locus: the incidence for one nonordinary contact has dimension at
most six, and adjoining each of the other two tact conditions adds at most
four parameters.  The resulting triple incidence therefore has dimension at
most fourteen.

The parameter spaces, incidence schemes, and their projections are defined
over \(\QQ\).  We compare the five Cartier divisors above with those used in
\cite[Proposition~2.1]{BLOS}.  There a circle is a conic through the two
circular points.  The common zero locus of these two point conditions is
\(V(\lambda_1,\lambda_2)\), while the three tangency conditions are the proper
transforms of the same three tact divisors cut out by
\(\tau_{Q_1},\tau_{Q_2},\tau_{Q_3}\).  Thus the geometric points parametrized
by our fiber are precisely the complete conics parametrized by the enumerative
intersection in the cited proposition.  We use this comparison only to obtain
a finite nonempty fiber; its scheme structure comes from the five local
equations above.  On the locus of smooth circles with ordinary contacts, Lemma
\ref{lem:ordinary-discriminant} further identifies this unmarked intersection
scheme with the incidence in which the three unique contact points are
recorded.

After base change to \(\CC\), \cite[Proposition~2.1]{BLOS} supplies a
nonempty open set of triples for which the corresponding set of geometric
points is finite and nonempty.  Hence the corresponding fiber of \(\mathcal Z\)
is finite and nonempty.  Intersecting this open set with the complements of
the two proper bad images above shows that the geometric generic fiber of
\(\mathcal Z\to B\) is finite and nonempty.  Upper semicontinuity of fiber
dimension gives a nonempty open of \(B_{\QQ}\) over which all fibers are
zero-dimensional.  Since the projection is proper, its restriction over this
open is finite.  The corresponding finite algebra has positive generic rank;
after shrinking once more, generic freeness makes it finite locally free of
positive rank.  We obtain a nonempty open
\(B^{\mathrm{fin}}_{\QQ}\subset B_{\QQ}\) over which the zero scheme is
finite and nonempty.  We use the enumeration here only for this conclusion;
the degree will be computed below.  Base change gives the required nonempty
open \(B^{\mathrm{fin}}\subset B\) over every characteristic-zero field.
Remove from \(B^{\mathrm{fin}}\) the closures of the two bad images.  On the
resulting nonempty open set \(B'\), every geometric zero lies in \(Y^\circ\)
and has three ordinary contacts, and \(\mathcal Z\times_BB'\) agrees
scheme-theoretically with \(\mathcal I^{\mathrm{ord}}\times_BB'\).  This
follows from the uniqueness of the contact point and Lemma
\ref{lem:ordinary-discriminant}.

The morphism \(\mathcal Z\times_BB'\to B'\) is proper and quasi-finite, hence
finite.  Its source is the open subspace
\(\mathcal I^{\mathrm{ord}}\times_BB'\) of the smooth, pure
fifteen-dimensional incidence \(\mathcal I^{\mathrm{ord}}\).  Every
irreducible component of the source therefore has dimension fifteen.  A
finite morphism preserves the dimension of an irreducible component and its
image.  Since \(B'\) is irreducible of dimension fifteen, every such component
dominates \(B'\).  In characteristic zero the resulting extensions of
function fields are separable, so the morphism is generically
\(\acute{e}\)tale.  Removing the branch loci of its finitely many components
gives a nonempty open \(B^{\mathrm{sg}}\subset B'\) over which \(\mathcal Z\)
is finite \(\acute{e}\)tale.  Every triple in this open set is strongly
general.

The length of a finite strongly general fiber is the top Chern number of
\(\mathcal E\).  For the blowup along the smooth center \(S\) of codimension
three, let \(j:E\hookrightarrow X\) and \(p:E=\PP(N_{S/P})\to S\) be the
natural maps, with \(\xi=c_1(\OO_E(1))\).  Since
\(j^*E=-\xi\) and \(p_*(1)=p_*(\xi)=0\), \(p_*(\xi^2)=1\), the projection
formula gives
\[
 \pi_*(E^j)=0\quad (j=1,2),
 \qquad
 \pi_*(E^3)=(-1)^{3-1}[S]=[S].
\]
Since \(S\subset\PP^5\) is the Veronese surface of degree four,
\[
 H^2(6H-2E)^3=216H^5-8H^2E^3=216-8\deg(S)=184.
\]
When \(k=\RR\), all these constructions are defined over \(\RR\).  The space
\(B_0\) of smooth conics is a nonempty Zariski-open subset of
\(\PP^5_{\RR}\).  Thus \(B=B_0^3\) is smooth and irreducible over \(\RR\),
and \(B(\RR)\) is Zariski dense in \(B\).  A proper Zariski-closed subset has
empty Euclidean interior in \(B(\RR)\).  Hence
\(B^{\mathrm{sg}}(\RR)\) is Euclidean dense in \(B(\RR)\).
\end{proof}

\begin{proposition}\label{prop:orientation}
The bundle \(\mathcal E\) is relatively orientable.  More precisely,
\[
        \omega_X\otimes\det\mathcal E
        \cong\OO_X(7H-2E)^{\otimes2}.
\]
\end{proposition}

\begin{proof}
By \eqref{eq:global-section},
\[
 \det\mathcal E\cong\OO_X(2H+3T)=\OO_X(20H-6E).
\]
Together with \eqref{eq:canonical}, this gives
\[
 \omega_X\otimes\det\mathcal E\cong\OO_X(14H-4E),
\]
which is the asserted square.
\end{proof}

Fix, once and for all, an isomorphism in Proposition
\ref{prop:orientation} and hence a relative orientation.  It determines an
\(\Aone\)-Euler number
\(n^{\Aone}(\mathcal E)\in\GW(k)\), and the Poincar\'e--Hopf formula expresses
it as the sum of the local indices of any section with isolated separable
zeros \cite{BachmannWickelgrenEuler}.

\begin{proof}[Proof of Theorem \ref{thm:introduction-quadratic}]
By Lemma \ref{lem:strongly-general}, the geometric degree of the zero scheme is
the complete-conic intersection number \(H^2T^3=184\).  Thus the rank of
\(n^{\Aone}(\mathcal E)\) is \(184\).

Write
\[
 \mathcal E=\mathcal F\oplus\mathcal G,
 \qquad
 \mathcal F=\OO_X(H)^{\oplus2},
 \qquad
 \mathcal G=\OO_X(T)^{\oplus3}.
\]
The summand \(\mathcal G\) has odd rank.  The odd-summand hyperbolicity theorem
\cite[Proposition 19]{SrinivasanWickelgren}, applied to this splitting and the
relative orientation of Proposition \ref{prop:orientation}, shows that
\(n^{\Aone}(\mathcal E)\) vanishes in the Witt group.  Equivalently, it is an
integral multiple of the hyperbolic plane.  The required section-independence
is \cite[Theorem 1.1]{BachmannWickelgrenEuler}.
Its rank is \(184\), and hence
\[
        n^{\Aone}(\mathcal E)=92\HHyp.
\]
The Poincar\'e--Hopf formula gives the stated identity.
\end{proof}

\subsection{Comparison with the complete-conic count}

For comparison, the enriched Chasles count of Bachmann and Wickelgren
\cite[Theorem~1.1]{BachmannWickelgrenExcess}, together with the preceding
theorem, gives
\begin{align*}
 e\bigl(X,\OO_X(T)^{\oplus5}\bigr)&=1632\HHyp,\\
 e\bigl(X,\OO_X(H)^{\oplus2}\oplus\OO_X(T)^{\oplus3}\bigr)
     &=92\HHyp.
\end{align*}
The second identity is obtained on the same space by replacing two tact
summands with the hyperplane summands cutting out the circle system.  Thus the
change from \(1632\HHyp\) to \(92\HHyp\) reflects the change of bundle rather
than a different compactification.

\section{Two-flag degenerations}\label{sec:degeneration}

In this section we study the two-input analogue of the three-flag degeneration
in \cite[Theorem~2.2]{BLOS}.  One conic remains smooth, and we retain the first
normal term of each degenerating tact equation; its sign determines whether
the four nearby solutions are real.

We begin on the affine circle chart.  Write
\begin{equation}\label{eq:circle-matrix}
 C_{a,b,c}:x^2+y^2-2ax-2by+c=0,
 \qquad
 \mathsf C(a,b,c)=
 \begin{pmatrix}1&0&-a\\0&1&-b\\-a&-b&c\end{pmatrix}.
\end{equation}
The squared radius is \(\rho=a^2+b^2-c\).  If \(M\) is a nonsingular
symmetric matrix defining a conic, its tact equation on this chart is
\begin{equation}\label{eq:tact-discriminant}
        D_M(a,b,c)=\Disc_\lambda\det(M+\lambda\mathsf C(a,b,c)).
\end{equation}

A flag consists of a point \(F=\alpha d\) on a line \(L\) through the origin,
where \(d,n\) is an oriented orthonormal basis and
\(L=\{n\cdot(x,y)=0\}\).  The two limiting conditions on a circle are
\begin{equation}\label{eq:flag-factors-general}
 P_F=2\alpha\,d\cdot(a,b)-\alpha^2-c,
 \qquad
 L_F=(d\cdot(a,b))^2-c.
\end{equation}
The first says that the circle passes through \(F\), and the second says that
it is tangent to \(L\).

For a parameter \(\delta\), smooth the flagged double line by
\begin{equation}\label{eq:flag-smoothing-general}
 h_{F,\delta}=(n\cdot(x,y))^2
 -\delta^2(d\cdot(x,y)-\alpha)^2+\delta^4.
\end{equation}
Its determinant is \(-\delta^6\), so it is a smooth conic for
\(\delta\ne0\).

\begin{lemma}\label{lem:normalized-tact}
There is a polynomial \(G_F(e)\) in \(e,a,b,c\), with real coefficients and
\(e=\delta^2\), such
that
\[
 D_{h_{F,\delta}}=e^2G_F(e)
 \qquad\text{and}\qquad
 G_F(0)=(P_FL_F)^2.
\]
In particular, if \(A_F=\partial_eG_F(0)\), then
\begin{equation}\label{eq:normalized-expansion}
        G_F(e)=(P_FL_F)^2+eA_F+O(e^2).
\end{equation}
\end{lemma}

\begin{proof}
An orthogonal affine change of coordinates reduces to
\(d=(1,0)\), \(n=(0,1)\).  The matrix of \eqref{eq:flag-smoothing-general}
is then
\[
 \begin{pmatrix}
 -e&0&e\alpha\\0&1&0\\e\alpha&0&e^2-e\alpha^2
 \end{pmatrix}.
\]
Substitution in \eqref{eq:tact-discriminant} shows that the discriminant is
divisible by \(e^2\).  After division, its specialization at \(e=0\) is
\[
 \bigl((a^2-c)(2\alpha a-\alpha^2-c)\bigr)^2.
\]
Returning to \(d,n\) gives the formula.
\end{proof}

Now take two flags.  Write their factors as \(P_i,L_i\), their normalized tact
equations as \(G_i(e)\), and their first coefficients as \(A_i\).  Fix a
smooth conic \(Q\), with tact equation \(D_Q\), and consider the reduced
scheme
\begin{equation}\label{eq:reduced-scheme}
 Z_{\mathrm{red}}=V(D_Q,P_1L_1,P_2L_2)\subset\Aaff^3_{a,b,c}.
\end{equation}

\begin{definition}\label{def:transverse-flag-datum}
The datum \((Q,F_1\in L_1,F_2\in L_2)\) is transverse if
\(Z_{\mathrm{red}}\) is finite and reduced and, at every \(z\in
Z_{\mathrm{red}}(\CC)\), the squared radius \(\rho(z)\) is nonzero, precisely
one factor \(r_i\in\{P_i,L_i\}\)
vanishes for each \(i\), the complementary factors \(s_i\) do not vanish,
\[
        dD_Q(z)\wedge dr_1(z)\wedge dr_2(z)\ne0,
\]
and \(A_1(z)A_2(z)\ne0\).
\end{definition}

\begin{theorem}\label{thm:fourfold-splitting}
Let a two-flag datum be transverse.  For every
\(z\in Z_{\mathrm{red}}(\CC)\), the special fiber of
\[
        D_Q=G_1(\delta^2)=G_2(\delta^2)=0
\]
has local algebra of length four at \(z\).  For every sufficiently small
nonzero complex \(\delta\), exactly four simple solutions converge to \(z\).

Suppose that the datum is real and \(z\) is real.  If
\begin{equation}\label{eq:real-splitting-sign}
        A_1(z)<0,\qquad A_2(z)<0,
\end{equation}
then all four solutions are real for sufficiently small positive
\(\delta\).  If either coefficient is positive, none of the four solutions
converging to \(z\) is real.
\end{theorem}

\begin{proof}
The transversality assumption makes \(D_Q,r_1,r_2\) a regular coordinate
system at \(z\).  On the hypersurface \(D_Q=0\), put \(u=r_1\) and
\(v=r_2\).  Since \(s_1(z)s_2(z)\ne0\), the two special equations are units
times \(u^2\) and \(v^2\).  Their completed local algebra is therefore
\[
        \CC[[u,v]]/(u^2,v^2),
\]
which has length four.

By \eqref{eq:normalized-expansion}, the two equations have the exact form
\[
 u^2U_1(u,v)+\delta^2A_1(u,v)+\delta^4B_1(u,v,\delta^2)=0,
\]
\[
 v^2U_2(u,v)+\delta^2A_2(u,v)+\delta^4B_2(u,v,\delta^2)=0,
\]
where \(U_i(0,0)=s_i(z)^2\) and the functions \(U_i,B_i\) are analytic near
the origin.  This follows by restricting \(G_i(0)=r_i^2s_i^2\) and the
polynomial expansion in \(e\) to \(D_Q=0\).
Set \(u=\delta x\), \(v=\delta y\) and divide by \(\delta^2\).  At
\(\delta=0\) the limiting equations are
\[
 s_1(z)^2x^2+A_1(z)=0,
 \qquad
 s_2(z)^2y^2+A_2(z)=0.
\]
They have four simple solutions over \(\CC\), because transversality includes
\(A_1(z)A_2(z)\ne0\).  The analytic implicit-function theorem produces one
simple branch from each of them.  It remains to exclude additional branches;
the length of the special fiber alone would not suffice for this step.

Choose a relatively compact coordinate neighborhood \(W\) in which the
origin is the only zero of the special fiber.  After shrinking \(W\), there
are constants \(m,M>0\) such that \(|U_i|\geq m\) and
\(|A_i|,|B_i|\leq M\) there for \(|\delta|\) small.  Every zero in \(W\)
therefore satisfies
\[
 |u|^2\leq \frac{M}{m}(|\delta|^2+|\delta|^4),
 \qquad
 |v|^2\leq \frac{M}{m}(|\delta|^2+|\delta|^4).
\]
Thus \(x=u/\delta\) and \(y=v/\delta\) remain bounded along every family of
zeros converging to the special fiber.  If additional zeros existed for a
sequence \(\delta_j\to0\), a subsequence of their rescaled coordinates would
converge to one of the four solutions of the limiting system.  The uniqueness
part of the implicit-function theorem would then identify those zeros with
the corresponding branch, a contradiction.  Hence the four displayed
branches are all the zeros converging to \(z\), and they are simple.  They also
lie in the ordinary tact locus for \(\delta\ne0\) sufficiently small.  Indeed,
along such a branch, with \(u=\delta x(\delta)\) and
\(v=\delta y(\delta)\), one has
\[
 \frac{\partial G_1}{\partial u}
   =2\delta x(0)s_1(z)^2+O(\delta^2),
 \qquad
 \frac{\partial G_2}{\partial v}
   =2\delta y(0)s_2(z)^2+O(\delta^2).
\]
The four limiting solutions have \(x(0)y(0)\ne0\), so both derivatives are
nonzero.  Moreover, \(dD_Q\ne0\) near \(z\), and \(\rho\ne0\) persists from
the special fiber.  Lemma \ref{lem:ordinary-discriminant} now shows that the
three contacts are ordinary.  Thus the four zeros represent smooth circles
tangent to the three smooth prescribed conics.

Over \(\RR\), uniqueness also shows that a branch issuing from a real limiting
solution is real for real \(\delta\).  The two values of \(x\) are real
exactly when \(A_1(z)<0\), and the same statement holds for \(y\).  Conversely,
any sequence of real zeros would have a real limit after rescaling.  Hence all
four branches are real exactly when both coefficients are negative, and none
is real if either coefficient is positive.
\end{proof}

\begin{corollary}\label{cor:saturation}
Suppose that a real transverse datum has \(N\) complex reduced points, of
which \(R\) are real and satisfy \eqref{eq:real-splitting-sign}.  Then every
sufficiently small positive smoothing has at least \(4R\) simple real tangent
circles and at least \(4(N-R)\) simple nonreal geometric tangent circles.  If
\(4N=184\), every Euclidean neighborhood of the smoothed triple
contains strongly general triples with exactly \(4R\) real tangent circles.
\end{corollary}

\begin{proof}
Theorem \ref{thm:fourfold-splitting} gives the stated branches in pairwise
disjoint neighborhoods of the reduced points.  They persist under small
changes of all three input conics.  By Lemma \ref{lem:strongly-general}, a
point of \(B^{\mathrm{sg}}(\RR)\) may be chosen arbitrarily close to the
smoothed triple.  It has exactly \(184\) complex tangent circles.  If
\(4N=184\), the
local branches already account for all of them.  Complex conjugation preserves
the separation between the real and nonreal neighborhoods, so the real count
remains \(4R\).
\end{proof}

\section{The 160-circle chamber}\label{sec:construction}

We now apply the preceding theorem to an explicit rational datum.  Let
\begin{equation}\label{eq:source-conic}
 q=x^2-\frac{25}{3}xy+\frac{600}{29}y^2-\frac{15}{2}x
 -\frac{2975}{29}y+\frac{25}{2}.
\end{equation}
The determinant of its symmetric homogeneous matrix is
\(-135734375/30276\), so \(q\) is smooth.
Put
\begin{equation}\label{eq:flag-data}
 \alpha=-\frac34,
 \qquad
 \beta=\frac{297}{323},
 \qquad
 d=\left(\frac{480}{481},\frac{31}{481}\right),
 \qquad
 n=\left(-\frac{31}{481},\frac{480}{481}\right).
\end{equation}
The two flags are
\[
 F_1=(\alpha,0)\in L_1=\{y=0\},
 \qquad
 F_2=\beta d\in L_2=\{n\cdot(x,y)=0\}.
\]
Their factors are
\begin{equation}\label{eq:explicit-factors}
\begin{aligned}
 P_1&=2\alpha a-\alpha^2-c,& L_1&=a^2-c,\\
 P_2&=2\beta(d_1a+d_2b)-\beta^2-c,&
 L_2&=(d_1a+d_2b)^2-c.
\end{aligned}
\end{equation}

We first describe the branches of the reduced problem.  The list is exhaustive
because, set-theoretically,
\[
 V(P_1L_1,P_2L_2)=
 \bigcup_{R_1\in\{P_1,L_1\}\,,\ R_2\in\{P_2,L_2\}}V(R_1,R_2).
\]
The intersections defined by \(P_1P_2\), \(P_1L_2\), and \(L_1P_2\) are
rational curves.  On the fourth intersection, the identity
\[
 a^2=(d_1a+d_2b)^2
\]
splits it into the two linear branches \(L_1L_2^+\) and \(L_1L_2^-\).
Since \(\alpha\beta d_2\ne0\), solving these equations gives all five
parametrizations below, with no omitted component:
\begin{equation}\label{eq:parametrizations}
\begin{array}{c|ccc}
 &a&b&c\\ \hline
P_1P_2&t&
\dfrac{2\alpha t-\alpha^2+\beta^2-2\beta d_1t}{2\beta d_2}
&2\alpha t-\alpha^2\\[5pt]
P_1L_2&\dfrac{t^2+\alpha^2}{2\alpha}&
\dfrac{t-d_1(t^2+\alpha^2)/(2\alpha)}{d_2}&t^2\\[5pt]
L_1P_2&t&\dfrac{(t^2+\beta^2)/(2\beta)-d_1t}{d_2}&t^2\\[5pt]
L_1L_2^+&t&\dfrac{1-d_1}{d_2}t&t^2\\[5pt]
L_1L_2^-&t&\dfrac{-1-d_1}{d_2}t&t^2.
\end{array}
\end{equation}
Each displayed map is an isomorphism from \(\Aaff^1\) onto the indicated
affine branch.  The inverse parameter is \(a\) on the branches
\(P_1P_2\), \(L_1P_2\), and \(L_1L_2^\pm\), and is
\(d_1a+d_2b\) on the branch \(P_1L_2\).  Thus substituting these five maps
into \(D_q\) counts every point of \(Z_{\mathrm{red}}\) and introduces no
multiple cover.

\begin{proposition}\label{prop:exact-profile}
The datum \eqref{eq:source-conic}--\eqref{eq:flag-data} is transverse.  The
five branches in \eqref{eq:parametrizations} have the following numbers of
complex roots and real roots:
\[
\begin{array}{c|ccccc}
 &P_1P_2&P_1L_2&L_1P_2&L_1L_2^+&L_1L_2^-\\ \hline
\text{complex roots}&6&12&12&8&8\\
\text{real points}&6&10&12&4&8.
\end{array}
\]
The squared radius \(\rho\) is nonzero at every one of the \(46\) complex
points.  At every real point it is positive, and
\(A_1<0\), \(A_2<0\).
\end{proposition}

\begin{proof}
Substitute the five parametrizations into the tact polynomial \(D_q\) and
take primitive numerators.  Section \ref{sec:certificate} reconstructs the
five integer polynomials and gives the Euclidean and Sturm--Tarski
calculations.  The polynomials have degrees
\(6,12,12,8,8\), are squarefree, and have respectively
\(6,10,12,4,8\) real roots.  The same calculation proves that \(\rho>0\),
\(A_1<0\), and \(A_2<0\) at every real root, and that the branch polynomial
is relatively prime to the specialized Jacobian, the two complementary
factors, \(A_1\), \(A_2\), and \(\rho\).

It remains only to compare the two components obtained from \(L_1=L_2=0\).
They meet at \((a,b,c)=(0,0,0)\), and direct substitution gives
\[
 D_q(0,0,0)=\frac{5967992159269140625}{6518301696}\ne0.
\]
 Thus the five branch schemes are mutually disjoint on \(D_q=0\).  All
 \(46\) geometric points are distinct, the reduced datum is transverse, and
 the asserted sign conditions hold.  The nonvanishing specialized Jacobian
 also gives \(dD_q\ne0\) at every point.  By Lemma
 \ref{lem:ordinary-discriminant}, a smooth point of the quartic discriminant
 represents a unique ordinary contact.  Thus every reduced circle has an
 ordinary tangency to \(q\), as required in the subsequent smoothing.
\end{proof}

For \(\delta>0\), define
\begin{align}
 h_{1,\delta}&=y^2-\delta^2(x-\alpha)^2+\delta^4,
 \label{eq:first-smoothing}\\
 h_{2,\delta}&=(n\cdot(x,y))^2
 -\delta^2(d\cdot(x,y)-\beta)^2+\delta^4.
 \label{eq:second-smoothing}
\end{align}

\begin{theorem}\label{thm:160-chamber}
For every sufficiently small positive \(\delta\), every Euclidean
neighborhood of the triple
\[
        (q,h_{1,\delta},h_{2,\delta})
\]
contains a strongly general triple with exactly \(160\) real tangent circles
and \(24\) nonreal geometric tangent circles, the latter forming \(12\)
complex-conjugate pairs.  The real circles have positive squared radius.  The
same real count holds on a nonempty Euclidean open subset of the strongly
general locus.
\end{theorem}

\begin{proof}
Proposition \ref{prop:exact-profile} gives \(46\) complex reduced points, of
which \(40\) are real.  All the hypotheses of Theorem
\ref{thm:fourfold-splitting} hold, and the two smoothing coefficients are
negative at every real point.  Thus the small smoothing has \(4\cdot40=160\)
simple real branches and \(4\cdot6=24\) simple nonreal geometric branches,
forming \(12\) conjugate pairs.  Since
\(4\cdot46=184\), these branches admit pairwise disjoint,
conjugation-invariant neighborhoods and persist over a Euclidean neighborhood
\(V_\delta\) of the smoothed triple.  Let \(N\) be any Euclidean neighborhood
of that triple.  By Lemma \ref{lem:strongly-general}, the intersection
\[
             N\cap V_\delta\cap B^{\mathrm{sg}}(\RR)
\]
is nonempty.  Choose a triple \(b\) in this intersection.  Its tangent-circle
scheme has geometric length \(184\), so the persisted local branches exhaust
the fiber and give exactly \(160\) real circles and \(24\) nonreal geometric
circles, forming \(12\) conjugate pairs.  Positivity of the squared radius is
open and holds at every limiting real point by
Proposition \ref{prop:exact-profile}.

Finally, the fiber over \(b\) is finite and \(\acute{e}\)tale.  After shrinking
about \(b\), the \(184\) roots remain in the chosen disjoint neighborhoods;
the real and nonreal neighborhoods remain separated, and the squared radii on
the real branches remain positive.  Intersecting this neighborhood with the
Zariski-open set \(B^{\mathrm{sg}}(\RR)\) gives the required nonempty
Euclidean open chamber.
\end{proof}

It remains to exhibit a rational point in this chamber.  Set
\begin{equation}\label{eq:fixed-delta}
                         \delta_0=\frac1{500000}.
\end{equation}
The equations \eqref{eq:source-conic}, \eqref{eq:first-smoothing}, and
\eqref{eq:second-smoothing} therefore give three conics over \(\QQ\), with no
unspecified perturbation.

\begin{theorem}\label{thm:explicit-160}
The triple
\[
                    (q,h_{1,\delta_0},h_{2,\delta_0})
\]
is strongly general.  Its tangent-circle scheme has \(184\) geometric points,
of which exactly \(160\) are real and have positive squared radius.  The
remaining \(24\) points form twelve complex-conjugate pairs.
\end{theorem}

\begin{proof}
First consider the input geometry.  If \(Q,H_1,H_2\) are the symmetric
homogeneous matrices of the three conics, then
\[
 \det Q=-\frac{135734375}{30276},\qquad
 \det H_1=\det H_2=-\delta_0^6.
\]
Thus all three prescribed conics are smooth.

Work in the affine circle chart with coordinates \((a,b,c)\), where the
circle matrix is
\[
       C(a,b,c)=
       \begin{pmatrix}1&0&-a\\0&1&-b\\-a&-b&c\end{pmatrix}.
\]
Let \(F_i(a,b,c)\) be the discriminant, with respect to \(\lambda\), of
\(\det(Q_i+\lambda C(a,b,c))\), with nonzero rational contents removed.  The
exact certificate described in Section \ref{sec:fixed-certificate} produces
\(184\) pairwise disjoint rational boxes after realification.  The rational
Krawczyk inclusion is strict on each box, so each contains a zero of
\((F_1,F_2,F_3)\), and the Jacobian determinant is nonzero throughout the
box.  Of these boxes, \(160\) lie in \(\RR^3\) and the other \(24\) are disjoint from
\(\RR^3\).  The squared radius \(a^2+b^2-c\) is positive on every real box
and is nonzero on every complex box.  Hence the certified points are
distinct smooth circles, precisely \(160\) of them real.

It remains to prove that the list is exhaustive on the space of complete
conics.  Homogenize the circle matrix to
\[
 C(z,a,b,c)=
 \begin{pmatrix}z&0&-a\\0&z&-b\\-a&-b&c\end{pmatrix}.
\]
On each of the charts \(z=0,a=1\) and \(z=0,b=1\), the three restricted tact
equations have reduced Gr\"obner basis \(\{1\}\) over \(\QQ\).  These two
charts cover the plane at infinity except for the double line
\([0:0:0:1]\).  Over that point the exceptional fiber is
\(\PP^2_{[u:v:w]}\).  Formula \eqref{eq:exceptional-tact-restriction} shows
that the proper tact section associated with a conic matrix \(M\) restricts,
up to a nonzero scalar, to
\begin{equation}\label{eq:fixed-boundary-quadratic}
 (um_{22}+wm_{11}-2vm_{12})^2
       -4(uw-v^2)(m_{11}m_{22}-m_{12}^2).
\end{equation}
For \(M=Q,H_1,H_2\), the three quadrics
\eqref{eq:fixed-boundary-quadratic} have reduced Gr\"obner basis \(\{1\}\)
on each of the charts \(u=1,v=1,w=1\).  There is consequently no zero on the
exceptional fiber either.

The complete-conic zero scheme is projective, and the preceding calculation
places it entirely in the affine circle chart.  It is therefore finite.  The
five defining equations form a parameter ideal in every local ring of the
smooth fivefold \(X\); since these local rings are Cohen--Macaulay, they form a
regular sequence.  The fundamental zero-cycle of this finite zero scheme is
therefore the top Chern class of \(\mathcal E\), and its length is
\[
                         H^2(6H-2E)^3=184
\]
by the intersection calculation in the proof of Lemma
\ref{lem:strongly-general}.  The \(184\) certified simple points
already have this total length, so there are no other geometric points and no
embedded contribution.  At every point the three tact differentials are
linearly independent.  In particular each tact divisor is smooth there;
Lemma \ref{lem:ordinary-discriminant} identifies this with ordinary tangency.
The zero scheme is thus finite and \(\acute{e}\)tale, consists of smooth
circles, is disjoint from the exceptional divisor, and has only ordinary
contacts.  This is precisely strong generality.
\end{proof}

Theorem \ref{thm:introduction-real} follows from Theorems
\ref{thm:160-chamber} and \ref{thm:explicit-160}.

\section{Local indices and trace forms}\label{sec:local-indices}

We next give an intrinsic description of the local indices.  Let \(K\) be a
field of characteristic zero, and let a smooth circle \(C\)
and a smooth conic \(Q\)
over \(K\) have a unique ordinary geometric contact.  Its point \(p\) is then
\(K\)-rational.  B\'ezout's theorem gives
\[
             C\cdot Q=2p+D_{\mathrm{res}},
\]
where \(D_{\mathrm{res}}\) is a reduced effective divisor of degree two,
disjoint from \(p\).  Its coordinate algebra is a quadratic \(\acute{e}\)tale
\(K\)-algebra.  We denote the discriminant square class of its trace pairing
by
\begin{equation}\label{eq:residual-discriminant}
 \delta_{\mathrm{res}}(C,Q)=
 \det\bigl(\Tr_{D_{\mathrm{res}}/K}(e_ie_j)\bigr)
 \quad\text{in }K^\times/(K^\times)^2.
\end{equation}
Changing the basis \((e_1,e_2)\) multiplies the determinant by a square.

Put \(T=T_pC=T_pQ\), \(N=T_p\PP^2/T\), and
\[
 \Delta_{Q,C}=\operatorname{II}_Q-\operatorname{II}_C
 \in\Hom(\Sym^2T,N).
\]
If \(v\) is a first-order deformation of \(C\), denote its normal
displacement at \(p\) by \(\eta_v(p)\in N\).

We shall use the following convention for square classes of differentials.  If
\(W\) is a \(K\)-vector space and \(\ell,m\in W^\vee\) are nonzero and
proportional, then
\[
             \ell\equiv u m\pmod{(K^\times)^2}
\]
means that \(\ell=vm\) for some \(v\in K^\times\) whose image in
\(K^\times/(K^\times)^2\) is the image of \(u\).  If \(u\) is itself a
square class, any representative may be used.

\begin{lemma}\label{lem:affine-orientation-comparison}
On the affine chart of \(P\) in which the coefficient of \(X_0^2\) is one,
write a conic as
\[
 X_0^2+\xi_1X_0X_1+(1+\xi_2)X_1^2
 -2aX_0X_2-2bX_1X_2+cX_2^2.
\]
Order the coordinates as \((\xi_1,\xi_2,a,b,c)\), take
\(\lambda_1=\xi_1\), \(\lambda_2=\xi_2\), and order the three tact summands
by \((Q_1,Q_2,Q_3)\).  There is a fixed class
\(\varepsilon_{\mathrm{aff}}\in K^\times/(K^\times)^2\), depending only on
these orders and the relative orientation fixed after Proposition
\ref{prop:orientation}, such that at every isolated zero \(C\) in this chart,
away from \(E\),
\begin{equation}\label{eq:affine-orientation-comparison}
 J_C^{\mathrm{or}}\equiv\varepsilon_{\mathrm{aff}}
 \det\frac{\partial(\tau_{Q_1},\tau_{Q_2},\tau_{Q_3})}
 {\partial(a,b,c)}
 \pmod{(K^\times)^2}.
\end{equation}
Here the tact sections are written in the local frames specified in the
proof.  The class \(\varepsilon_{\mathrm{aff}}\) is independent of the zero
and of the matrix representatives of the \(Q_i\).  Over \(\RR\), replacing
\(c\) by the positive-radius coordinate
\(c=a^2+b^2-r^2\), \(r>0\), contributes the fixed sign \(-1\).
\end{lemma}

\begin{proof}
Let \(e_H\) be the frame of \(\OO_X(H)\) supplied by the coefficient of
\(X_0^2\), and let \(e_E\) be the canonical section of \(\OO_X(E)\), which is
invertible away from \(E\).  We use the ordered frames
\[
 e_H,e_H,e_T,e_T,e_T,\qquad e_T=e_H^6e_E^{-2},
\]
for the five summands of \(\mathcal E\), and
\(e_L=e_H^7e_E^{-2}\) for \(\OO_X(7H-2E)\).  Together with
\[
 d\xi_1\wedge d\xi_2\wedge da\wedge db\wedge dc,
\]
these frames give a local relative orientation.  Its ratio with the fixed
isomorphism
\(\omega_X\otimes\det\mathcal E\cong\OO_X(7H-2E)^{\otimes2}\)
is a unit on \(\Aaff^5\setminus S\).  Since \(\Aaff^5\) is normal and the
omitted Veronese surface has codimension three,
\[
 \Gamma(\Aaff^5\setminus S,\OO)^\times=K^\times.
\]
The resulting square class is therefore a single constant
\(\varepsilon_{\mathrm{aff}}\).

In the specified input and output orders, the Jacobian of the section is
block triangular:
\[
 \frac{\partial(\lambda_1,\lambda_2,
 \tau_{Q_1},\tau_{Q_2},\tau_{Q_3})}
 {\partial(\xi_1,\xi_2,a,b,c)}
 =
 \begin{pmatrix}
 I_2&0\\ *&
 \dfrac{\partial(\tau_{Q_1},\tau_{Q_2},\tau_{Q_3})}
 {\partial(a,b,c)}
 \end{pmatrix}.
\]
This proves \eqref{eq:affine-orientation-comparison}.  Replacing a matrix for
\(Q_i\) by \(u_iQ_i\) multiplies its tact row by \(u_i^6\), a square, so the
class is independent of the representatives.  Finally,
\[
 \det\frac{\partial(a,b,c)}{\partial(a,b,r)}=-2r.
\]
For \(r>0\), its square class over \(\RR\) is \(-1\), which proves the last
assertion.
\end{proof}

\begin{proposition}\label{prop:intrinsic-differential}
Choose \(t\in T\setminus\{0\}\) and \(n^\vee\in N^\vee\setminus\{0\}\).
Let \(W=T_{[C]}P\) be the first-order deformation space of \(C\).  For the
discriminant-normalized local equation of the proper tact divisor, the two
nonzero linear forms on \(W\) satisfy
\begin{equation}\label{eq:intrinsic-differential}
 d\tau_Q|_C\equiv
 2\delta_{\mathrm{res}}(C,Q)
 n^\vee\!\left(\Delta_{Q,C}(t,t)\right)
 n^\vee\!\left(\eta_{(-)}(p)\right)
 \pmod{(K^\times)^2}.
\end{equation}
The square class is independent of the choices of \(t\) and \(n^\vee\).

Let \(C\) now be a real zero for a strongly general real triple
\(Q_\bullet\).  Write \(\delta_i=\delta_{\mathrm{res}}(C,Q_i)\).  The three
contact points are real and the squared radius of \(C\) is positive.  Write
the center as \(c\), take the positive radius \(r>0\), and put
\(n_i=(p_i-c)/r\).  Measure the signed curvatures of both \(C\) and \(Q_i\)
with respect to this normal.  There is a constant
\(\varepsilon_0\in\{\pm1\}\), depending only on the fixed relative
orientation and the ordered affine coordinates, such that
\begin{equation}\label{eq:curvature-jacobian}
 J_C^{\mathrm{or}}\equiv
 \varepsilon_0\prod_{i=1}^3
 \delta_i(\kappa_{Q_i}-\kappa_C)
 \det\begin{pmatrix}
 n_{1x}&n_{1y}&1\\
 n_{2x}&n_{2y}&1\\
 n_{3x}&n_{3y}&1
 \end{pmatrix}
 \pmod{(\RR^\times)^2}.
\end{equation}
Moreover, \(\delta_i>0\) if the two residual points of \(C\cap Q_i\) are
real, whereas \(\delta_i<0\) if they form a nonreal conjugate pair.  The same
\(\varepsilon_0\) works for every such circle and every strongly general real
triple in this affine chart.
\end{proposition}

\begin{proof}
For a degree-\(n\) polynomial we use
\(\Disc(f)=a_n^{2n-2}\prod_{i<j}(r_i-r_j)^2\), and use its homogeneous
counterpart for binary forms.  After a projective change of coordinates,
write
\[
 C:XZ-Y^2=0,
 \qquad [s:t]\longmapsto[s^2:st:t^2].
\]
If \(M_Q=(m_{ij})\) and \(M_C\) are the corresponding symmetric matrices and
\(F_Q(s,t)=Q(s^2,st,t^2)\), direct expansion gives
\begin{equation}\label{eq:pencil-binary-discriminant}
 \Disc_\lambda\det(M_Q+\lambda M_C)
       =2^{-8}\Disc_{s,t}(F_Q).
\end{equation}
The factor \(2^{-8}\) is a square.  Under a change of parameter in
\(\operatorname{GL}_2(K)\), the discriminant of a binary quartic is multiplied
by the twelfth power of the determinant.  Under a congruence by
\(g\in\operatorname{GL}_3(K)\), the determinant pencil is multiplied by
\(\det(g)^2\), and its cubic discriminant is therefore multiplied by
\(\det(g)^8\).  Rescaling either conic also contributes an even power.
Consequently all coordinate changes used above alter
\eqref{eq:pencil-binary-discriminant} only by squares, so the comparison is
intrinsic at the level claimed in the statement.

 Since \(C\) is smooth, its point on \(X\) lies outside the exceptional
 divisor.  If \(e_E\) is a local equation for \(E\) in compatible local frames,
 then the equation of the proper tact divisor is
 \[
        \tau_Q=e_E^{-2}\Disc_\lambda\det(M_Q+\lambda M_C).
 \]
 Here \(e_E(C)\) is a unit.  Because the discriminant vanishes at \(C\),
 differentiation at \(C\) gives
 \[
        d\tau_Q|_C=e_E(C)^{-2}
        d\!\left(\Disc_\lambda\det(M_Q+\lambda M_C)\right)\big|_C.
 \]
 The factor \(e_E(C)^{-2}\) is a square.  Hence passing from the pencil
 discriminant to the normalized equation of the proper tact divisor does not
 change the square class of its differential.

Let \(U\) be the local deformation space of \(C\) in the open set of smooth
conics, pointed by
 \(0=[C]\), and let \(\mathcal C\to U\) be the universal conic.  After
 replacing \(U\) by an \(\acute{e}\)tale neighborhood of \(0\), the smooth
 orbit map from \(\operatorname{PGL}_3\) to the space of smooth conics
 trivializes \(\mathcal C\) as \(C\times U\).  Since the contact point is
 \(K\)-rational, \(C\cong\PP^1_K\).  Choose an affine coordinate \(x\), with
 tangent vector \(t\) at \(p\), whose point at infinity is outside
 \(2p+D_{\mathrm{res}}\).  Shrinking \(U\), the leading coefficient of the
 restricted quartic is a unit.  Division by this coefficient gives a family
 \(F\) of monic quartics.  This normalization does not change the square class
 of the differential: multiplying a quartic by a unit \(u\) multiplies its
 discriminant by \(u^6\), and at a zero of the discriminant
 \[
                 d(u^6\Disc(F))=u(0)^6\,d\Disc(F).
 \]

 The central monic quartic has a factorization \(F_0=f_0g_0\), where \(f_0\)
 has the double zero \(p\) and \(g_0\) cuts out \(D_{\mathrm{res}}\).  These
 factors are coprime.  The differential of the multiplication map
 \[
  \Aaff^2\times\Aaff^2\longrightarrow\Aaff^4,
  \qquad(f,g)\longmapsto fg,
\]
on the spaces of monic quadratic polynomials has determinant
 \(\operatorname{Res}(f_0,g_0)\ne0\).  The \(\acute{e}\)tale inverse-function
 theorem therefore gives a unique factorization \(F=fg\) into monic
 deformations of \(f_0\) and \(g_0\).  Equivalently, the relative
 intersection divisor splits into a degree-two contact part and a degree-two
 residual part.  This factorization is valid over the dual numbers and may be
 used to compute the differential of the discriminant.

To identify the differential of the contact factor, work in the completed
local ring of \(\PP^2\) at \(p\).  Choose a normal parameter \(y\)
 with \(dy=n^\vee\), and use the preceding \(x\) as tangent parameter.  The
 formal implicit-function theorem writes \(C\), \(Q\), and every first-order
 deformation of \(C\) as graphs over the \(x\)-axis.  Modulo the square of the
 deformation ideal and terms of order at least three in \(x\), the graph of
 the deformed circle minus the graph of \(Q\) is
 \[
  a(v)+b(v)x-\frac{\gamma}{2}x^2,
  \qquad
  \gamma=n^\vee\!\left(\Delta_{Q,C}(t,t)\right).
 \]
 Here \(a(0)=b(0)=0\), and the definition of normal displacement gives
 \(da(v)=n^\vee(\eta_v(p))\).  Formal Weierstrass preparation identifies the
 degree-two contact factor with the displayed quadratic to first order.
 Ordinary tangency is equivalent to \(\gamma\ne0\), and its discriminant is
 \(b(v)^2+2\gamma a(v)\) modulo the square of the deformation ideal.  Its
 differential at the central fiber is therefore \(2\gamma\,da\).  The product
 formula
\[
 \Disc(fg)=\Disc(f)\Disc(g)\operatorname{Res}(f,g)^2
\]
and the \(\acute{e}\)tale factorization just constructed show that the
differential of the quartic discriminant has square class
\[
 2\delta_{\mathrm{res}}(C,Q)\gamma\,da.
\]
The discriminant of \(g_0\) represents the discriminant of the trace pairing
 of the quadratic \(\acute{e}\)tale algebra of \(D_{\mathrm{res}}\); changing
 the affine coordinate, the monic normalization, or a basis multiplies the
 resulting differential by a square.  Equation
 \eqref{eq:pencil-binary-discriminant} now gives
\eqref{eq:intrinsic-differential}.  Rescaling \(x\), \(y\), \(t\), or
\(n^\vee\) changes the expression by a square.  In particular, rescaling
\(n^\vee\) changes both scalar factors involving it and hence contributes a
square.  This also proves the asserted independence of choices.

Complex conjugation shows that the unique ordinary contact of a real circle
with a real conic is real.  A real contact on a smooth real circle forces
\(\rho>0\), so the positive radius used in the statement is defined.  The
trace form of a split quadratic algebra has positive discriminant, while the
trace form of \(\CC/\RR\) has negative discriminant.  This proves the assertion
about the sign of \(\delta_i\).

For a circle variation \((\dot c,\dot r)\), the normal displacement at
\(p_i\) is, up to the common sign determined by whether displacement of the
circle or of its defining function is used,
\(n_i\cdot\dot c+\dot r\).  Apply
\eqref{eq:intrinsic-differential} to the three tangencies and take the
determinant of the resulting rows.  Choose a Euclidean unit tangent \(t_i\) at
\(p_i\), and identify the normal line with \(\RR n_i\).  By the definition of
the second fundamental form,
\[
 n_i^\vee\!\left(\Delta_{Q_i,C}(t_i,t_i)\right)
   =\kappa_{Q_i}-\kappa_C.
\]
For an arbitrary nonzero tangent, the right side is multiplied by its squared
length, a positive square.  Thus this normalization does not change the
square class.  The factor \(2^3\) is also positive.

Lemma \ref{lem:affine-orientation-comparison} compares this determinant, in
the stated row order, with the global oriented Jacobian by one square class
independent of the zero and of the input conics.  Passing from \((a,b,c)\) to
the positive-radius coordinates \((a,b,r)\) contributes the fixed sign
\(-1\).  Absorbing these two fixed factors into \(\varepsilon_0\) proves
\eqref{eq:curvature-jacobian}.
\end{proof}

We now pass from the local formula to the trace forms.  Let \(Q_\bullet\) be a
strongly general real triple, and let
\(A=\Gamma(Z(\sigma_{Q_\bullet}),\OO)\) be its finite \(\acute{e}\)tale real
algebra.  The ordinary trace form is
\[
        H(f,g)=\Tr_{A/\RR}(fg).
\]
At a simple closed zero with separable residue field, the local
\(\Aone\)-degree is the transfer of the one-dimensional form defined by this
oriented Jacobian \cite{KassWickelgrenLocal}.
The relative orientation gives a square class
\(J\in A^\times/(A^\times)^2\), whose value at a real point has the sign in
\eqref{eq:curvature-jacobian}.  Choose a representative \(j\in A^\times\) and
define the weighted trace form
\[
        H^{\mathrm{or}}(f,g)=\Tr_{A/\RR}(jfg).
\]
This is well defined up to isometry: if \(j'=ju^2\), multiplication by \(u\)
identifies the two forms.

For a nondegenerate real symmetric form \(K\), write
\(\inertia(K)=(\nu_+(K),\nu_-(K))\).

\begin{theorem}\label{thm:trace-forms}
Suppose that the triple has \(R\) real tangent circles and \(c\) nonreal
conjugate pairs.  Let \(n_+\), \(n_-\) be the numbers of positive and negative
oriented local indices.  Then
\[
 \inertia(H)=(R+c,c),
 \qquad
 \inertia(H^{\mathrm{or}})=(c+n_+,c+n_-).
\]
Moreover,
\[
        n_+=n_-=\frac R2,
        \qquad
        \inertia(H^{\mathrm{or}})=(92,92).
\]
\end{theorem}

\begin{proof}
There is an isomorphism of real \(\acute{e}\)tale algebras
\[
        A\cong\RR^R\times\CC^c.
\]
On a real factor the ordinary form is \(\langle1\rangle\), while the weighted
form is \(\langle j(C)\rangle\).  On a complex factor either trace form has
inertia \((1,1)\).  This proves the first two formulas.  The signature of
Theorem \ref{thm:introduction-quadratic} is zero, so
\(n_+-n_-=0\).  Since \(n_++n_-=R\) and \(R+2c=184\), the remaining
assertions follow.
\end{proof}

\begin{corollary}\label{cor:hermite-160}
In the chamber of Theorem \ref{thm:160-chamber},
\[
        \inertia(H)=(172,12),
        \qquad
        \inertia(H^{\mathrm{or}})=(92,92).
\]
The inequality \(R\leq136\) is equivalent to
\[
        \nu_-(H)\geq24.
\]
Thus it is an ordinary real-root condition and does not follow from the
orientation-weighted Euler class.
\end{corollary}

\section{Exact arithmetic for the five branches}\label{sec:certificate}

In this section we complete the proof of Proposition
\ref{prop:exact-profile} by recording its exact finite certificate.  All
identities and sign counts below are over \(\QQ\); the code and frozen data
that reproduce them are archived in \cite{LiuEtAlCode}.  For a branch
\(B\) in \eqref{eq:parametrizations}, let
\(\phi_B:\Aaff^1\to\Aaff^3\) be its parametrization and put
\begin{equation}\label{eq:branch-polynomial}
 p_B(t)=\operatorname{pp}\bigl(D_q(\phi_B(t))\bigr),
\end{equation}
where \(\operatorname{pp}\) denotes the primitive numerator with the sign
obtained by clearing the positive denominators in the displayed formulas.
The exact formula \eqref{eq:branch-polynomial}, together with
\eqref{eq:tact-discriminant}, \eqref{eq:source-conic}, and
\eqref{eq:parametrizations}, determines these polynomials over \(\QQ\).  The
accompanying checker computes them directly from these data.

We next record a certificate for the root and sign assertions.  Let
\(f\in\QQ[t]\) be squarefree and let \(g\in\QQ[t]\).  Put
\[
 q_{f,g}=\operatorname{rem}(f'g,f)
\]
and let \(\operatorname{Sturm}(f;g)\) be the signed remainder sequence
\[
 f,q_{f,g},-\operatorname{rem}(f,q_{f,g}),\ldots,
\]
terminating at the last nonzero remainder.  If \(\epsilon_-(f;g)\) and
\(\epsilon_+(f;g)\) are its sign strings at \(-\infty\) and \(+\infty\),
respectively, put
\[
 I(f;g)=V\bigl(\epsilon_-(f;g)\bigr)
       -V\bigl(\epsilon_+(f;g)\bigr).
\]
Here \(V\) denotes the number of sign variations.  Sturm's theorem gives the
number of real roots of \(p_B\) as \(I(p_B;1)\).  If \(g\) has no common
root with \(p_B\), the Sturm--Tarski theorem
\cite[Theorem~2.61]{BasuPollackRoy} gives
\begin{equation}\label{eq:sturm-tarski-query}
 I(p_B;g)=
 \sum_{p_B(t)=0,\ t\in\RR}\operatorname{sgn}(g(t))
 =\TaQ(g,p_B).
\end{equation}
For a rational function, we replace it by the product of its numerator and
denominator; this has the same sign wherever it is defined.

Let \(\rho_B=\rho\circ\phi_B\), and let
\(A_{i,B}=A_i\circ\phi_B\).  Clearing positive constant denominators, the
Euclidean algorithm gives the following table.  In each sign-string entry,
the strings before and after the comma are the signs at \(-\infty\) and
\(+\infty\).

\begin{center}
\scriptsize
\setlength{\tabcolsep}{2.5pt}
\begin{tabular}{c|c|c|c|c}
\toprule
\(B\)&\(g\)&\(\epsilon_-,\epsilon_+\)&\((V_-,V_+)\)&\(I(p_B;g)\)\\
\midrule
\(P_1P_2\)&\(1\)&\(-+-+-+-,-------\)&\((6,0)\)&6\\
&\(\rho_B\)&\(-+-+-+-,-------\)&\((6,0)\)&6\\
&\(-A_{1,B}\)&\(-+-+-+-,-------\)&\((6,0)\)&6\\
&\(-A_{2,B}\)&\(-+-+-+-,-------\)&\((6,0)\)&6\\
\addlinespace
\(P_1L_2\)&\(1\)&\(-+-+-++-+-+-+,------+++++++\)&\((11,1)\)&10\\
&\(\rho_B\)&\(-+-++-+-+-+-+,----+++++++++\)&\((11,1)\)&10\\
&\(-A_{1,B}\)&\(-+-+-++-+-+-+,------+++++++\)&\((11,1)\)&10\\
&\(-A_{2,B}\)&\(-+-+--+-+-+-+,-----++++++++\)&\((11,1)\)&10\\
\addlinespace
\(L_1P_2\)&\(1\)&\(-+-+-+-+-+-+-,-------------\)&\((12,0)\)&12\\
&\(\rho_B\)&\(-+-+-+-+-+-+-,-------------\)&\((12,0)\)&12\\
&\(-A_{1,B}\)&\(-+-+-+-+-+-+-,-------------\)&\((12,0)\)&12\\
&\(-A_{2,B}\)&\(-+-+-+-+-+-+-,-------------\)&\((12,0)\)&12\\
\addlinespace
\(L_1L_2^+\)&\(1\)&\(+-++-++-+,+++---+++\)&\((6,2)\)&4\\
&\(\rho_B\)&\(+--+-++-+,++----+++\)&\((6,2)\)&4\\
&\(-A_{1,B}\)&\(+--+-++-+,++----+++\)&\((6,2)\)&4\\
&\(-A_{2,B}\)&\(+--++-+-+,++--+++++\)&\((6,2)\)&4\\
\addlinespace
\(L_1L_2^-\)&\(1\)&\(+-+-+-+-+,+++++++++\)&\((8,0)\)&8\\
&\(\rho_B\)&\(+-+-+-+-+,+++++++++\)&\((8,0)\)&8\\
&\(-A_{1,B}\)&\(+-+-+-+-+,+++++++++\)&\((8,0)\)&8\\
&\(-A_{2,B}\)&\(+-+-+-+-+,+++++++++\)&\((8,0)\)&8\\
\bottomrule
\end{tabular}
\end{center}

The same Euclidean calculation gives
\begin{equation}\label{eq:all-gcds}
 \gcd(p_B,p_B')=1
\end{equation}
for every branch, and
\begin{equation}\label{eq:nonvanishing-gcds}
 \gcd\bigl(p_B,h_B\bigr)=1
\end{equation}
for each of the following primitive specialized numerators:
\[
 h_B\in\left\{
 \rho_B, A_{1,B}, A_{2,B},
 \left(\det\frac{\partial(D_q,r_1,r_2)}{\partial(a,b,c)}\right)\!\circ\phi_B,
 s_1\circ\phi_B, s_2\circ\phi_B
 \right\}.
\]
With each gcd made monic, the output needed below is
\[
\begin{array}{c|ccccccc}
 B&p_B'&\rho_B&A_{1,B}&A_{2,B}&J_B&s_{1,B}&s_{2,B}\\ \hline
 P_1P_2&1&1&1&1&1&1&1\\
 P_1L_2&1&1&1&1&1&1&1\\
 L_1P_2&1&1&1&1&1&1&1\\
 L_1L_2^+&1&1&1&1&1&1&1\\
 L_1L_2^-&1&1&1&1&1&1&1
\end{array}
\]
where \(J_B\) denotes the specialized Jacobian in
\eqref{eq:nonvanishing-gcds} and \(s_{i,B}=s_i\circ\phi_B\).
Equations \eqref{eq:all-gcds}--\eqref{eq:nonvanishing-gcds} are identities
in \(\QQ[t]\), verified by the ordinary polynomial Euclidean algorithm;
all their inputs are explicitly determined by
\eqref{eq:tact-discriminant}, \eqref{eq:normalized-expansion},
\eqref{eq:source-conic}, \eqref{eq:explicit-factors}, and
\eqref{eq:parametrizations}.  Thus the calculation is reproducible from the
formulas in the paper alone.

The rows with \(g=1\) give respectively
\(6,10,12,4,8\) real roots.  In every branch, the three queries for
\(\rho_B,-A_{1,B},-A_{2,B}\) equal the total number of real roots.  Since the
gcds in \eqref{eq:nonvanishing-gcds} exclude zero values, every summand in
each query is \(+1\).  Hence \(\rho_B>0\), \(A_{1,B}<0\), and
\(A_{2,B}<0\) at every real root.  The remaining gcds prove transversality
and exclude the complementary factors.  Finally, the two \(L_1L_2\) branches
meet only at the origin and
\[
 D_q(0,0,0)=\frac{5967992159269140625}{6518301696}\ne0.
\]
Consequently the five squarefree branch schemes are disjoint and have total
geometric length \(6+12+12+8+8=46\), with exactly forty real points.  This
proves every assertion used in Proposition \ref{prop:exact-profile}.

\subsection{The fixed rational fiber}\label{sec:fixed-certificate}

We now describe the exact certificate used in Theorem \ref{thm:explicit-160}.
This is logically separate from the limiting Sturm--Tarski calculation above:
it verifies the single rational value \(\delta_0=1/500000\).

For a polynomial map \(F:\CC^3\to\CC^3\), write it as a map
\(F_{\RR}:\RR^6\to\RR^6\).  Let \(X=x_0+[-r,r]^m\) be a rational box and let
\(A\in\operatorname{GL}_m(\QQ)\).  All interval operations below use
the elementary midpoint--radius rules with rational endpoints.  We form
\begin{equation}\label{eq:krawczyk-operator}
 K(X)=x_0-AF_{\RR}(x_0)+
        \bigl(I-A\,DF_{\RR}(X)\bigr)(X-x_0).
\end{equation}
If \(K(X)\subset\operatorname{int}X\), the Krawczyk theorem
\cite[Section~13]{RumpVerification} gives a zero of \(F_{\RR}\) in \(X\).
If the interval determinant of \(DF_{\RR}(X)\) avoids zero, every zero in the
box is simple.  This is all that is needed below; exhaustion by the
intersection number will imply uniqueness in each box.  For a real box we
apply the same criterion with \(m=3\).

For the three tact equations of the fixed triple, exact rational evaluation of
\eqref{eq:krawczyk-operator} gives \(160\) real boxes of radius \(10^{-55}\)
and \(24\) realified complex boxes of radius \(10^{-50}\).  The worst ratios
of the coordinate radii of \(K(X)-x_0\) to those of \(X-x_0\) are, respectively,
\[
              1.322\cdot10^{-10}
              \qquad\text{and}\qquad
              4.978\cdot10^{-21}.
\]
Thus every inclusion is strict by a wide margin.  On the real boxes, the
certified lower bound for \(a^2+b^2-c\) is
\(8.7628\cdot10^{-4}\), and the certified distance of the Jacobian determinant
from zero is \(1.1977\cdot10^{-6}\).  The \(24\) complex boxes have certified
distance \(1.2830\cdot10^{-2}\) from the locus \(a^2+b^2-c=0\), while the
realified Jacobian determinant has distance at least \(1.0830\cdot10^{31}\)
from zero.  Their \(a\)-coordinate rectangles, together with those of the real
boxes, are pairwise disjoint; the smallest separating margin is
\(1.2460\cdot10^{-8}\).  The nonreal \(a\)-rectangles are separated from the
real axis by at least \(0.3249\).  These inequalities are displayed only in
decimal form for readability; the archived checker uses the exact rational
numbers.

There are two independent global checks.  The first is the boundary
calculation in the proof of Theorem \ref{thm:explicit-160}.  It consists of
five reduced Gr\"obner bases over \(\QQ\), one on each of the charts
\[
 z=0, a=1;\qquad z=0, b=1;\qquad
 u=1;\qquad v=1;\qquad w=1,
\]
and all five bases are \(\{1\}\).  The second is a Dixon elimination of
\(b,c\).  The specialized Dixon matrix has size \(35\), rank \(33\) over
\(\QQ(a)\), and a fixed \(33\)-by-\(33\) minor has primitive determinant
\(P(a)\) satisfying
\[
             \deg P=184,\qquad \gcd(P,P')=1.
\]
Exact continued-fraction isolation gives \(160\) real roots of \(P\).  The
polynomial and its exact checker are archived in the accompanying repository.
The proof of exhaustion uses the boundary calculation and the intersection
number \(184\), not an assumption that an arbitrary Dixon minor is free of
extraneous roots.  The \(160\) real roots found by Dixon isolation agree with
the real-box count.

\subsection{Accompanying code and data}

The accompanying repository \cite{LiuEtAlCode} contains the exact branch
generator and its frozen Sturm--Tarski certificate; the rational centers,
preconditioners, and interval checkers for the \(160\) real and \(24\) nonreal
boxes; the five boundary computations; and the degree-\(184\) Dixon
eliminant.  The proof-critical checking paths use only integer or rational
polynomial arithmetic.  Numerical iteration is used only to produce candidate
centers, which are subsequently certified by exact rational inequalities.
The repository also contains pinned software versions, reproduction
instructions, expected reports, and a manifest of file digests.

\end{document}